\documentclass{amsart}
\usepackage{amsfonts}
\usepackage{amsmath}
\usepackage{graphicx, tikz, subcaption}
\usepackage{color}
\usepackage[shortlabels]{enumitem}
\newtheorem{theorem}{Theorem}
\newtheorem{proposition}{Proposition}

\newtheorem{lemma}{Lemma}
\newtheorem{defn}{Definition}
\newtheorem{cor}{Corollary}
\newtheorem{rem}{Remark}
\newcommand{\eps}{\varepsilon}
\newcommand{\ds}{\displaystyle}

\title{Twisted torus knots with Horadam parameters}

\author{Brandy Doleshal}
\address{Brandy Doleshal
\newline Sam Houston State University
\newline Huntsville, TX
\newline USA}
\email{bdoleshal@shsu.edu}
\urladdr{}

\begin{document}

\title{Twisted torus knots with Horadam parameters}

\author{Brandy Doleshal}
\address{Sam Houston State University \\
Box 2206 \\
Huntsville, TX, 77341}

\maketitle

\begin{abstract}
Sangyop Lee has done much work to determine the knot types of twisted torus knots, including classifying the twisted torus knots which are the unknot. Among the unknotted twisted torus knots are those of the form $K(F_{n+2}, F_n, F_{n+1}, -1)$, where $F_i$ is the $i$th Fibonacci number. Here, we consider twisted torus knots with parameters that are defined recursively, similarly to the Fibonacci sequence. We call these \textit{Horadam parameters}, after the generalization of the Fibonacci sequence introduced by A.F. Horadam. Here, we provide families of twisted torus knots that generalize Lee's work with Horadam parameters. Additionally, we provide lists of primitive/primitive and primitive/Seifert twisted torus knots and connect these lists to the Horadam twisted torus knots.
\end{abstract}

\keywords{twisted torus knots, Horadam sequences, primitive/primitive knots, primitive/Seifert knots}

\section{Introduction}

Twisted torus knots were introduced by Dean \cite{DeanSFSSurgeries} with the goal of studying knots with Seifert fibered space surgeries. Since then, others have viewed the twisted torus knots as interesting as a class of knots and have studied their knot types \cite{krishnamorton}, \cite{LeeTTKUnknots}, \cite{LeeTTKTypes}, \cite{LeeTTKTorus}, \cite{LeedePaiva}, their bridge spectra \cite{BowmanBridgeSpectra}, their ribbon lengths \cite{KimRibbonLength}, their knot Floer homology \cite{VafaeeKnotFloer}, and their knot polynomials \cite{adnanpark}, \cite{BavierDoleshal}, \cite{MortonAlexPoly}.

In particular, Lee has spent a great deal of effort to determine which twisted torus knots are unknots, torus knots and satellite knots \cite{LeeTTKUnknots}, \cite{LeeTTKTypes}, \cite{LeeTTKTorus}, also collaborating with de Paiva on these efforts \cite{LeedePaiva}. Among this work, Lee \cite{LeeTTKUnknots} shows that the twisted torus knots contain a family with parameters in the Fibonacci sequence that are all unknotted. Here we find a family of twisted torus knots that also all have the same knot type, all of whose parameters are consecutive terms in some Horadam sequence, a generalization of the Fibonacci sequence. We call knots with this type of parameters \textit{Horadam twisted torus knots}. We further explore the Horadam twisted torus knots, and make connections with work of Kadokami \cite{Kadokami}.

In \cite{DeanSFSSurgeries}, Dean discusses some requirements for twisted torus knots to be primitive/primitive and primitive/Seifert. Here we provide the list of primitive/primitive knots that are twisted torus knots and show that not all of these are apparently Horadam twisted torus knots. We then correct and expand work of Dean, connecting the primitive/Seifert twisted torus knots to the Horadam twisted torus knots.

This work is organized into four sections. In Section \ref{background}, we provide the necessary definitions and background. In Section \ref{Horadam}, we provide some lemmas about Horadam sequences. In Section \ref{HTTKs}, we consider the twisted torus knots with Horadam parameters, proving Proposition \ref{prop:knottypes}. Finally, in Section \ref{primplus}, we provide lists of twisted torus knots that are primitive/primitive and primitive/Seifert, proving Theorems \ref{thm:pp}, \ref{thm:phS}, and \ref{thm:pS}. Further, we consider their connections to the Horadam twisted torus knots.

\section{Background and definitions}\label{background}
We begin by defining the twisted torus knots. Let $p$ and $q$ be relatively prime integers that are both at least 2, let $m$ and $n$ be relatively prime nonzero integers, and let $r$ be an integer with $2 \le r \le p+q$. A \textit{twisted torus knot} $K(p,q,r,m,n)$ is constructed from a torus knot $T(p,q)$ in the following way. First consider a disk $D$ in the torus which intersects $T(p,q)$ in $r$ adjacent parallel arcs, as pictured in Figure \ref{fig:diskwithtorusknot}. Consider another torus containing $r$ parallel copies of the torus knot $T(m,n)$ and a disk, $D'$, intersecting the given link $r$ times, one for each component of the link. We excise $D$ and $D'$ from their respective tori , so that we have created two once-punctured tori. We glue the two punctured tori along the boundaries of $D$ and $D'$ so that the $2r$ intersection points with $T(p,q)$ and the $r$ copies of $T(m,n)$ match and provide a coherently oriented curve. 
\begin{center}
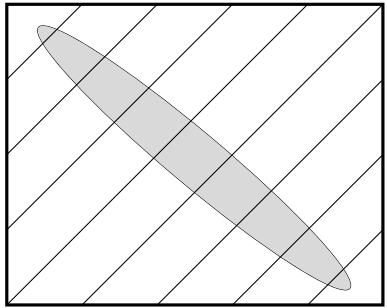
\begin{figure}
	\begin{tikzpicture}
	 	\draw[rotate = -40] (.65,3.1) ellipse (2.7 and 0.35);
		\fill[rotate = -40, gray!30] (.65,3.1) ellipse (2.7 and 0.35);
		\draw[very thick] (0,0) rectangle (5,4);
		\draw (0,3) -- (1,4);
		\draw (0,2) -- (2,4);
		\draw (0,1) -- (3,4);
		\draw (0,0) -- (4,4);
		\draw (1,0) -- (5,4);
		\draw (2,0) -- (5,3);
		\draw (3,0) -- (5,2);
		\draw (4,0) -- (5,1);
	\end{tikzpicture}\caption{$T(5,3)$ with $D$ intersecting 8 strands}\label{fig:diskwithtorusknot}
\end{figure}
\end{center}

We note here that the twisted torus knot is a curve in the genus 2 Heegaard surface for $S^3$. This Heegaard surface bounds two genus 2 handlebodies, which we denote by $H$ and $H'$.

Lee \cite{LeeTTKTypes} shows that if $|m| \ge 2$ and $|n| \ge 2$, then $K(p,q,r,m,n)$ is a satellite knot, so if we are interested in hyperbolic knots, we can focus on the case when $m = \pm 1$ or $n = \pm 1$. Often, $m$ is considered to be 1, and we suppress the $m$ in that case, calling the twisted torus knot $K(p,q,r,n)$. Lee \cite{LeeTTKTypes} shows that switching $p$ and $q$ does not change the knot type, so we also typically assume that $p >q$.
\begin{lemma}[Lee \cite{LeeTTKTypes}]\label{lem:lee}
$K(p,q,r,n)$ and $K(q,p,r,n)$ have the same knot type.
\end{lemma}

In this case, when $r \le p$, we can consider the twisted torus knot $K(p,q,r,n)$ as the closure of a braid on $p$ strands given by the braid word $$(\sigma_{p-1} \cdots \sigma_2 \sigma_1)^q (\sigma_{r-1} \cdots \sigma_2 \sigma_1)^n,$$
so we often use braids to help our understanding of twisted torus knots. Here the symbol $\sigma_i$ represents the braid shown in Figure \ref{fig:sigmai}, where the crossing is between the $i$ and $i+1$ strands.

When $p < r \le p+q$, we consider the case where the disk $D$ intersects more than $p$ strands, which can also be considered as the closure of a braid, though slightly more complicated. To determine the corresponding braid when $r = p+q$, we consider Figure \ref{fig:ttkpplusq}, which shows a $(p,q)$-torus knot. The thick bar, which we call a ``switch", indicates that $p$ strands both enter and leave the switch, but that the strands change from one grouping to another in the switch. The number next to the strands represents the number of parallel strands running in parallel with no crossings.

In this figure, $q$ strands are pulled over slightly to intersect the disk $D$, shown in gray, referenced in the definition of a twisted torus knot. In this picture, we can see that if we consider the dot to be the braid axis, we have a braid on $p+q$ strands. Further, all $p+q$ strands will intersect $D$, which will create $n$ full twists on the $p+q$ strands when identified with $D'$. After the strands leave the disk, the leftmost $q$ strands pass under the rightmost $p$ strands. Hence, in the case when $r = p+q$, $K(p,q,p+q, n)$ can be written as the closure of the braid 
shown in Figure \ref{fig:r=p+qbraid}, where the number in the box represents the number of times $\sigma_{p+q - 1} \cdots \sigma_1$ appears in the braid word.

In this work, we will use the braid representations of twisted torus knots extensively. When it is convenient, we will represent full twists on a number of strands as a link surgery diagram.

\begin{center}
\begin{figure}[h!]
	\begin{tikzpicture}
		\draw (0,0) -- (0,1);
		\draw (.5,0) -- (.5,1);
		\node at (1,.5) {$\cdots$};
		\draw (1.5,0) -- (1.5,1);
		\draw (2.5,0) .. controls (2.45,.5)and (2.05,.5) .. (2,1);	
		\fill[white] (2.15,.45) rectangle (2.35,.55);
		\draw (2,0) .. controls (2.05,.5)and (2.45,.5) .. (2.5,1);		
		\draw (3,0) -- (3,1);
		\node at (3.5,.5) {$\cdots$};
		\draw (4,0) -- (4,1);
		\draw (4.5,0) -- (4.5,1);
	\end{tikzpicture}\caption{$\sigma_i$}\label{fig:sigmai}
\end{figure}
\end{center}

\begin{center}
\begin{figure}
	\begin{tikzpicture}
 		\fill[gray!30, rotate = -45] (.65,-1.35) ellipse (.15 and .75);
	 	\draw[thick] (0,0) ellipse (4 and 2.5);
		\draw[thick] (0,.2) ellipse (1.25 and .5);
		\fill[white] (-1.3,.75) rectangle (1.3,.2);
		\draw[thick] (-1.1,-.05) .. controls (-.75,.3)and (.75,.3) .. (1.1,-.05);	
	  	\draw (0,1.25) .. controls (-1.5, 1.25) and (-2.95,0) .. (-2.25,-.75);
		\node[scale = .85] at (-1.95, 1.0) {$p$};
		\draw[thick] (-2.5, -1.0) -- (-2, -.5);
		\draw[thick] (0, .95) -- (0, 1.55);
		\draw (-2.35,-.85) .. controls (-2.45, -.95) and (0,-2) .. (0,-2.5);
		\node[scale = .85] at (-1.95, -1.5) {$q$};
		\draw (-2.15,-.65) .. controls (0, -2.25) and (2.5,-1) .. (2.5,0);
		\node[scale = .85] at (-1.25, -.65) {$p-q$};
		\draw[dotted] (0,-2.5) -- (.75,-.2);
		\draw (.75,-.2) .. controls (-2, -1.75) and (2.2,-1) .. (2.2,0);
		\draw (2.5,0) .. controls (2.5, .95) and (1.5,1.5) .. (0,1.4);
		\draw (2.2,0) .. controls (2.2, .75) and (1.5,1.1) .. (0,1.1);
		\draw[thick, rotate = -45] (.65,-1.35) ellipse (.15 and .75);
		
		\filldraw[red] (.35,-.65) circle (1.25pt);
\end{tikzpicture}\caption{Constructing $K(p,q,p+q,n)$}\label{fig:ttkpplusq}
\end{figure}
\end{center}

\begin{center}
\begin{figure}[ht]
	\begin{tikzpicture}
		\draw[thick] (1.5,1) -- (1.5, 2); 
		\node at (2.1,1.7) {$p+q$}; 
		\draw[thick] (0.5,0) rectangle (2.5,1);
		\node at (1.5,.5) {$n(p+q)$};  
		\draw[thick] (1,0) .. controls (1,-.4)and (1.4,-1.1) .. (1.45,-1.1);
		\node at (0.7,-0.7) {$q$}; 
		\draw[thick] (2,0) .. controls(2,-1) and (1,-1.5) .. (1,-2);
		\node at (2.3,-0.7) {$p$}; 
		\draw[thick] (1.55,-1.25) .. controls (1.7,-1.5) and (2,-1.8) .. (2,-2);
	\end{tikzpicture}\caption{$K(p,q,p+q, n)$ as a braid}\label{fig:r=p+qbraid}
\end{figure}
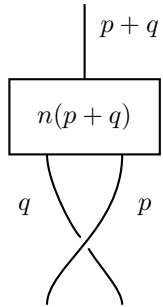
\end{center}

\section{Horadam Sequences}\label{Horadam}

The Fibonacci sequence has been studied extensively and several generalizations of it have been established and studied. Here we use the generalization introduced by Horadam \cite{Horadam2} and studied by countless others since.

\begin{defn}
An $(m, n; a,b)$-Horadam sequence $\mathcal H$ is the sequence $\{ H_0, H_1, H_2 \ldots \}$ where $H_0 = m$, $H_1 = n$ and $H_k = aH_{k-2}+bH_{k-1}$ for $k \ge 2$. \end{defn}

As an example, the Fibonacci sequence is the $(0,1; 1,1)$-Horadam sequence, which we will denote by $\mathcal F = \{ F_0, F_1, F_2, \ldots \}$. In this work, we focus mainly on the $(m,n; 1, 1)$-Horadam sequences, and we will shorten the notation to $(m,n)$-Horadam sequence, or $\mathcal H_{m,n}$. In \cite{Horadam2}, $m$ and $n$ are arbitrary integers, but in this work, we will at times need to restrict $m$ and $n$ to the positive integers that are relatively prime.

Here we state some lemmas proved by Horadam \cite{Horadam1} and provide some results about Horadam sequences that will be useful for our knot-theoretic results.

\begin{lemma}[Horadam \cite{Horadam1}]\label{lem:seq}
If $\mathcal H$ is the $(m,n)$-Horadam sequence, then $$H_k = mF_{k-1} + nF_k$$ for $k \ge 1$, where $F_i$ is the $i$th Fibonacci number.
\end{lemma}

Using the previous lemma, Horadam \cite{Horadam1} finds a relationship between any two terms in an $(m,n)$-Horadam sequence.

We let $s = m^2+mn -n^2$, which is the negative of Horadam's $e$ \cite{Horadam1}.

\begin{lemma}\label{lem:sloperelations}
For an $(m,n)$-Horadam sequence $\mathcal H$ and for $k \ge 1$ , 
\begin{enumerate}
\item $H_{k}^2 +H_{k+1}H_{k} - H_{k+1}^2 = (-1)^{k}s$, 
\item $\ds H_{k}^2 +H_{k}H_{k-1} - H_{k-1}^2 = (n^2+mn -m^2) + 2\varepsilon_k s  + 2\sum_{i=1}^{k-1} H_i^2$, and
\item $\ds H_{k}^2 +H_{k}H_{k-1} + H_{k-1}^2 = (n^2+mn -m^2) + 2\varepsilon_k s  + 2\sum_{i=1}^{k-1} H_i^2 + 2H_{k-1}^2$
\end{enumerate}
where $\varepsilon_k = 1$ if $k$ is even and $\varepsilon_k = 0$ if $k$ is odd.

\end{lemma}

\begin{proof}
Part 1 is given by Horadam \cite{Horadam1}, up to shift in terms that results in a negative on the right side of the equation, accounting for the fact that $s = -e$.

We prove part 2 by induction. For the base case, we let $k=1$, so we have $H_1^2 + H_1H_0 - H_0^2 = n^2 + mn - m^2$, by the definition of $\mathcal H$. On the other hand $\varepsilon_1 = 0$ because $1$ is odd and $\ds \sum_{i=1}^{0} H_i^2 $ is an empty sum, so the right side of the equation is also $ n^2 + mn - m^2$. When $k=2$, we have $H_{2}^2 +H_{2}H_{1} - H_{1}^2 = (m+n)^2 + n(m+n)- n^2 = m^2 + 3mn + n^2$ on the left of the equation, and $\ds (n^2+mn -m^2) + 2\varepsilon_2 s + 2\sum_{i=1}^{1} H_i^2  = n^2 + mn -m^2 + 2m^2 + 2mn -2n^2 + 2n^2 = m^2 + 3mn + n^2$ on the right, as expected.

Let $s_i = H_{i}^2 +H_{i}H_{i-1} - H_{i-1}^2$ for $i \ge 1$, and suppose the statement is true for $s_{k-1}$ for some $k\ge 2$. That is, our inductive hypothesis tells us that $$\ds s_{k-1} = (n^2+mn -m^2) + 2\varepsilon_{k-1} s  + 2\sum_{i=1}^{k-2} H_i^2.$$ Consider $s_k = H_{k}^2 +H_{k-1}H_{k} - H_{k-1}^2$. Using the definition of $\mathcal H$, we have that $H_{k-1} = H_{k} -H_{k-2}$ and $H_{k} = H_{k-1} + {H_{k-2}}$. Then we can rewrite $s_k$ as $H_{k}^2 + H_{k-1}^2 + H_{k-1}H_{k-2} - (H_k - H_{k-2})^2 = H_{k-1}^2 + H_{k-1}H_{k-2} -H_{k-2}^2+2H_{k}H_{k-2} $. This expression is equal to $s_{k-1} + 2 H_k H_{k-2}.$

 Part 1 tells us that $H_{k}^2 +H_{k+1}H_{k} - H_{k+1}^2 = (-1)^{k}s$. Using the definition of $\mathcal H$, we then have that $H_{k}H_{k-2}  = (-1)^{k-2} s + H_{k-1}^2$, so that $s_k = s_{k-1} + 2(-1)^{k-2} s + 2H_{k-1}^2$. Then, using the inductive hypothesis, we have 
\begin{align*}
s_k &= (n^2+mn -m^2) + 2\varepsilon_{k-1} s  + 2\sum_{i=1}^{k-2} H_i^2 +2(-1)^{k-2} s + 2H_{k-1}^2\\
& =(n^2+mn -m^2) + 2(\varepsilon_{k-1}  + (-1)^{k-2})s + 2\sum_{i=1}^{k-1} H_i^2 .
\end{align*}

If $k$ is even, then $k-1$ is odd and $\varepsilon_{k-1}  + (-1)^{k-2} = 1 = \varepsilon_k$. If $k$ is odd, then $k-1$ is even and $\varepsilon_{k-1}  + (-1)^{k-2} = 0 = \varepsilon_k$. Hence, the second part is proven.

The third part of the statement is obtained by adding $2H_{k-1}^2$ to both sides of the equation in the previous part of the lemma.
\end{proof}

The previous lemma tells us that for all $k \ge 2$, $H_{k-1}^2 +H_{k-1}H_{k} - H_{k}^2$ is the same value, up to sign. In the next two lemmas, we show that for the very similar-looking $H_{k}^2 +H_{k-1}H_{k} - H_{k-1}^2$ forms an increasing sequence so that every value is different.

\begin{lemma}\label{lem:increasing}
For $m$ and $n$ positive, let $\mathcal H$ be the $(m,n)$-Horadam sequence, and for $k \ge 2$, let $s_k = H_{k}^2 +H_{k-1}H_{k} - H_{k-1}^2$ and let $t_k = H_{k+1}^2 +H_{k+1}H_{k} + H_{k}^2$. The sequences $\{s_k \}_{k\ge 2}$ and $\{t_k \}_{k\ge 1}$ are both increasing. 
\end{lemma}

\begin{proof}

In the proof of Lemma \ref{lem:sloperelations}, we showed that $s_k = s_{k-1} + 2(-1)^{k-2} s + 2H_{k-1}^2$. Rearranging, we see that $s_k - s_{k-1} = 2(-1)^{k-2} s + 2H_{k-1}^2$, so to show the sequence is increasing, it is enough to show that $(-1)^{k-2} s + H_{k-1}^2$ is positive for all $k\ge 3$. That is, we want to show that $H_{k-1}^2 > (-1)^{k-1}s$. 

When $k$ is even, this inequality is $H_{k-1}^2 > n^2-mn-m^2$. Because $\mathcal H$ is an increasing sequence for positive values of $m$ and $n$, $H_{k-1}^2\ge H_1^2 = n^2$. Since $m$ and $n$ are positive, $n^2-mn-m^2<n^2 \le H_{k-1}^2$.

When $k$ is odd, the inequality becomes $H_{k-1}^2 > m^2+mn-n^2$. Again, $\mathcal H$ is an increasing sequence, so $H_{k-1}^2 \ge H_2^2 = m^2 + 2mn +n^2$, which is evidently greater than $m^2+mn-n^2$ for positive values of $m$ and $n$.

To show that $\{t_k \}_{k\ge 1}$ is increasing, we use the recursion relation for $\mathcal H$ to rewrite $t_k = H_{k+1}^2 +H_{k+1}H_{k} + H_{k}^2$ as $t_k = H^2_k + H_{k}H_{k-1} + H_{k-1}^2 + 2 H_{k+1}H_{k}$, which we recognize as $t_k = t_{k-1} + 2H_{k+1}H_k$. Then $t_k - t_{k-1} = 2H_{k+1}H_k$, which is at least 2 for all $k \ge 2$, so $\{t_{k}\}_{k\ge1}$ is increasing.
\end{proof}

\begin{rem}\label{rem:extendlemma}
We note here that if $m<n$, the sequence $\{s_k\}$ is increasing for $k\ge1$ as well since $H_{k-1}^2>H_0^2 = m^2$ and $m^2+mn-n^2 < m^2$. 
\end{rem}

\begin{lemma}\label{lem:slopesaredifferent}
For $m$ and $n$ positive, let $\mathcal H$ be the $(m,n)$-Horadam sequence. 
\begin{enumerate}
\item For $k \ge 2$, let $s_k = H_{k}^2 +H_{k-1}H_{k} - H_{k-1}^2$. For all $l \ge 1$, if $l\neq k$, then $s_l\neq s_k$. In fact, $|s_l - s_k|>1$.
\item For $k \ge 1$, let $t_k = H_{k+1}^2 +H_{k+1}H_{k} + H_{k}^2$. For all $l \ge 1$, if $l\neq k$, then $t_l\neq t_k$. In fact, $|t_l - t_k|>1$.
\end{enumerate}
\end{lemma}

\begin{proof}
The assertion that $s_l \neq s_k$ for $l \neq k$ is nearly an immediate result of Lemma \ref{lem:increasing}, with the exception of the case $l=1$. That is, we must show that $s_1 \neq s_k$ for $k\ge 2$. In Remark \ref{rem:extendlemma}, we noted that the $\{s_k\}_{k\ge1}$ is increasing when $m<n$, so we need only show that $s_1 \neq s_k$ when $m>n$. 

Suppose, for the sake of contradiction, that $m>n$ and there exists a $k\ge 2$ such that $s_k = s_1 = n^2 + mn - m^2$. Then we have  $$(n^2+mn -m^2) + 2\varepsilon_k s  + 2\sum_{i=1}^{k-1} H_i^2 = n^2+mn -m^2.$$ 

Simplifying this equation, we have $$  \varepsilon_k s   =  -  \sum_{i=1}^{k-1} H_i^2.$$ 

The right side of the equation is negative for $k\ge 2$, but $\varepsilon_k s = \varepsilon_k(m^2 + mn - n^2)$, which is at least zero because $m>n$. 

Since $\{s_k\}_{k\ge2}$ is an increasing integer sequence, to show that the difference between any two terms is greater than 1, we need only show that the difference between consecutive terms is greater than 1.  
As noted in the proof of Lemma \ref{lem:sloperelations}, $s_k - s_{k-1} = 2((-1)^{k-2} s + H_{k-1}^2)$ for $k\ge 2$. In the proof of Lemma \ref{lem:increasing}, we showed that $(-1)^{k-2} s + H_{k-1}^2$ is positive for $k\ge 2$, so $s_k - s_{k-1} \ge 2$ for these values of $k$.

To finish the proof, we must show that the absolute value of the difference between $s_2$ and $s_1$ is at least 2. Since $s_2 - s_1 = 2m(m+n) $ and $m$ and $n$ are positive, this number is always at least 2.

The assertion that $t_l \neq t_k$ for $l \neq k$ is an immediate result of Lemma \ref{lem:increasing}, and the proof that $|t_l-t_k| >1$ is similar to the proof for $\{s_k\}$, with the exception that $t_k - t_{k-1} = 2H_{k+1}H_k$, which is always at least 2. 
\end{proof}

Next, we prove two lemmas about when the $(m,n)$-Horadam sequence is a subset of a $(\pm 1,a)$-Horadam sequence for some integer $a \ge 2$.

\begin{lemma}\label{lem:subseq}
Suppose $1<  m < n$ and $\gcd(m,n) = 1$. Let $\{q_l, \ldots, q_1, q_0\}$ and $\{r_l, \ldots, r_1, r_0\}$ be the sequences of quotients and remainders, respectively, provided by the Euclidean algorithm for $m$ and $n$.  The $(m,n)$-Horadam sequence $\mathcal H_{m,n}$ is a subsequence of a $(\pm 1,a)$-Horadam sequence $\mathcal H_{\pm 1, a}$, for some integer $a\ge 2$, if and only if $q_i = 1$ for $1 \le i \le l$ and $q_0 \in\{1, 2\}$.
\end{lemma}

Because a pair of numbers with the property given in Lemma \ref{lem:subseq} has the largest ($q_0 = 1$), or nearly the largest ($q_0 = 2$), possible number of equations in its Euclidean algorithm, we call $(m,n)$ an \textit{maximal pair} if the pair has this property. 

\begin{proof}
First we prove that the given conditions guarantee that $\mathcal H_{m,n}$ is a subsequence of $\mathcal H_{\pm 1, a}$. Let the Euclidean algorithm for $m$ and $n$ be given by the sequence of equations 
 \begin{align*} 
n &=   q_l m + r_l\\ 
m &=  q_{l-1} r_l + r_{l-1} \\
r_l &= q_{l-2} r_{l-1} + r_{l-2} \\
& \vdots \\
r_2 &= q_0 r_1 + r_0
\end{align*}
Because $m$ and $n$ are coprime, we know that $r_0 = 1$. If $q_i = 1$ for $0\le i \le l$, then we have the sequence $\mathcal H_{1, r_1}$ so that $H_0 = r_0 = 1$, $H_1 = r_1$ and $H_i = r_i = r_{i - 1} + r_{i-2}$ for $2 \le i \le l$. Then $m = H_{l+1}$ and $n = H_{l+2}$, and we have that $\mathcal H_{m,n}$ is a subsequence of $\mathcal H_{1, r_1}$.

When $q_i = 1$ for $1\le i \le l$ and $q_0 = 2$, then we consider the sequence $\mathcal H_{-1, r_1 + 1} = \{-1, r_1 + 1, r_1,  2r_1 + 1, (2r_1 + 1) + r_1, \ldots \} = \{-1, r_1 + 1, r_1,  r_2, r_3, \ldots \}$. Again $m = H_{l+1}$ and $n = H_{l+2}$, and we have that $\mathcal H_{m,n}$ is a subsequence of $\mathcal H_{-1, r_1 + 1}$.

Now suppose that $\mathcal H_{m,n}$ is a subsequence of $\mathcal H_{1, a}$, so that $m$ and $n$ are consecutive terms in $\mathcal H_{1, a}$. That is $m = H_j $ and $n = H_{j+1} $ for some $j\ge 1$. From our assumptions, we have $n = m + H_{j-1}$ and $n = q_l m + r_l$ for some positive integers $q_l$ and $r_l$ with $1\le r_l \le m-1$. Combining these equations with the recursive definition of a Horadam sequence, we have  $$H_{j-1} = (q_l - 1) H_j + r_l .$$ Because $\mathcal H_{1, a}$ is an increasing sequence, $H_{j} > H_{j-1} $. In order to satisfy this inequality, we must have $q_l = 1$ and $r_l = H_{j-1}$. 

Now we turn our attention to the equation $$m = q_{l-1} r_l + r_{l-1}.$$ Since $m = H_j = H_{j-1} + H_{j-2}$ and $r_l = H_{j-1}$, we substitute and subtract $H_{j-1}$ from both sides to obtain $$H_{j-2} = (q_{l-1}-1) H_{j-1} + r_{l-1}.$$ 
Again, $\mathcal H_{1,a}$ is an increasing sequence, so $q_{l-1} = 1$ and $r_{l-1}= H_{j-2}$.

Consider the next equation from the Euclidean algorithm, we have $$r_l = q_{l-2} r_{l-1} + r_{l-2}.$$ Using $r_l = H_{j-1}$ and $r_{l-1} = H_{j-2}$, along with recursion and subtraction technique used above, we obtain the equation $$H_{j-3} = (q_{l-2}-1) H_{j-2} + r_{l-2},$$ which again gives $q_{l-2} = 1$ and $r_{l-2} = H_{j-3}$. 

We continue this process for each equation in the sequence of equations from the Euclidean algorithm, finding at each step that $$H_{j-i} = (q_{l-i+1}-1) H_{j-i+1} + r_{l-i+1}, $$ so that $q_{l-i+1}=1$ and $r_{l-i+1} = H_{j-i}$ for $1\le i \le j-1$. When $i = j$, we have $$H_{0} = (q_{l-j+1}-1) H_{1} + r_{l-j+1}, $$ 

We know that each of $r_0$ and $H_0$ is the unique element of its respective sequence that is equal to 1, so we see that the index $l-j+1$ must be 0, and in fact, $j = l+1$. Then $q_0 =1 $ as well. That is, all of the terms in the sequence of quotients $\{q_l,  \ldots, q_1, q_0\}$ are equal to 1. 

Suppose instead that $\mathcal H_{m,n}$ is a subsequence of $\mathcal H_{- 1, a}$, so that $m$ and $n$ are consecutive terms in $\mathcal H_{-1, a}$. That is $m = H_j $ and $n = H_{j+1} $ for some $j\ge 1$.  Once again we have $n = m + H_{j-1}$ and $n = q_l m + r_l$, so $$H_{j-1} = (q_l - 1) H_j + r_l .$$ Because $a \ge 2$, $\mathcal H_{-1, a}$ is very close to being an increasing sequence and $H_{j} > H_{j-1}$ unless $j = 2$, so the methods used when $\mathcal H_{m,n}$ is a subsequence of $\mathcal H_{1,a}$ need modification. We solve the equation for $r_l$ to obtain 
$$r_l = H_{j-1} - (q_l - 1)H_j .$$ By Lemma \ref{lem:seq}, $H_j = aF_j - F_{j-1}$ and $H_{j-1} = aF_{j-1} - F_{j-2}$, where $F_i$ is the $i$th Fibonacci number. Then we can rewrite $r_l$ as 
\begin{align*}
r_l &= aF_{j-1} - F_{j-2} - (q_l - 1)(aF_j - F_{j-1}) \\
&=  aF_{j-1} - F_{j-2} - (q_l - 1)(a(F_{j-1}+F_{j-2}) - F_{j-1}) \\
& = (a(2-q_l) +q_l - 1) F_{j-1} - (a(q_l-1) +1)F_{j-2} 
\end{align*}

Because $r_l > 0$ and $(a(q_l-1) +1)F_{j-2} \ge 0$, we have, at the very least, that $a(2-q_l) +q_l - 1 >0$. Then $$q_l - 1>a(q_l -2).$$
We also know that $a \ge 2$, so $$a(q_l - 2) \ge 2q_l -4.$$ Combining these two inequalities, we find that $ q_l < 3$, so $q_l \in \{1, 2\}$.  If $q_l = 2$, then $r_l = H_{j-1} - H_j \le 1$, with equality if and only if $j = 2$. Hence $j = 2$ and $l = 0$, which means only $q_0$ can be 2. If $q_l = 1$, then $r_l = H_{j-1}$. As above, we can iterate this process so that, for $1\le i \le j-2$, at the $i$th step we have $$H_{j-i} = (q_{l-i+1}-1) H_{j-i+1} + r_{l-i+1}. $$ Solving for $r_{l-i+1}$, using Lemma \ref{lem:seq} and rearranging, we have $$r_{l-i+1} = (a(2-q_{l-i+1}) +q_{l-i+1} - 1) F_{j-i} - (a(q_{l-i+1}-1) +1)F_{j-i+1}  .$$

Similarly to above, we find that the coefficient of $F_{j-i}$, $a(2-q_{l-i+1}) +q_{l-i+1} -1$, must be positive so that $q_{l-i+1} \in \{1,2\}$.  If $q_{l-i+1} = 2$, then $$r_{l-i+1} = H_{j-i} - H_{j-i+1},$$ which can only be positive if $j = i+1$, and in that case $r_{l-i+1} = 1 = r_0$, so $q_{l-i+1} = q_0 = 2$. If $q_{l-i+1} = 1$, then $r_{l-i+1} = H_{j-i}$, and we consider the next equation. To summarize, for each $i$ with $1 \le i \le j-2$, either the process terminates with $q_0 = 2$, or $q_{l-i+1} = 1$ and there are more equations to consider in the Euclidean algorithm for $m$ and $n$. Up until this point, each step has been dependent on the fact that the Fibonacci sequence increases, but $F_1 = F_2$, so we consider the case $i = j-1$ separately.

Suppose $i = j-1$, so the equation becomes $$H_{1} = (q_{l-j+2}-1) H_{2} + r_{l-j+2}, $$ which we can rewrite as $$r_{l-j+2} = a - (q_{l-j+2}-1) (a-1).$$
Because $r_{l-j+2}$ is at least 1, $q_{l-j+2} \in \{1,2\}$.  If $q_{l-j+2} = 2$, then $r_{l-j +2} = 1  = r_0$ so $l = j-2$ and $q_{l-j+2} = 2 = q_0$. If $q_{l-j+2} = 1$, then $r_{l-j+2} = H_1$. This means that the relevant equation of the Euclidean algorithm is $$ 2a-1 = (a-1) + a.$$  While this equation does hold, it cannot be an equation in the Euclidean algorithm for $m$ and $n$ because $a > a-1$. Hence $q_{l-j+2} = 2 = q_0$, and the process terminates. Thus the sequence  $\{q_l, \ldots q_1, q_0 \}$ satisfies $q_i =1$ for $1\le i \le l$ and $q_0 = 2$.

\end{proof}

\begin{lemma}\label{lem:maxpairs}
Suppose $1<  m < n$ and $\gcd(m,n) = 1$. Then $(m,n)$ is a maximal pair if and only if any two consecutive terms of $\mathcal H_{m,n}$ form a maximal pair.
\end{lemma}

\begin{proof}
Let $\{q_l, \ldots, q_0\}$ be the sequence of quotients for the Euclidean algorithm for $m$ and $n$.
Consider two consecutive terms $H_j$ and $H_{j+1}$ of $\mathcal H_{m,n}$ for $j \ge 1$. Because of the recursion relation for $\mathcal H_{m,n}$, we know that the $i$th equation of the Euclidean algorithm for $H_{j}$ and $H_{j+1}$ is $$H_{j+1 - i} = H_{j-i} + H_{j-1-i}$$ for $0 \le i \le j-1$, so the first $j$ terms in the sequence of quotients for $H_{j}$ and $H_{j+1}$ will be 1. If $i = j$, we have $H_2 = n + m$, and the next equation in the Euclidean algorithm for $H_j$ and $H_{j+1}$ is  $$ n = q_l m + r_l, $$
which begins the Euclidean algorithm for $m$ and $n$. Hence the sequence of quotients for $H_{j}$ and $H_{j+1}$ is 
 $\{q_{l + j+1}, \ldots, q_{l+1}, q_l \ldots, q_0\} = \{1, 1, \ldots, 1, q_l \ldots, q_0\} $. That is, $(m,n)$ is a maximal pair if and only if $(H_j, H_{j+1})$ is a maximal pair.

\end{proof}

Now we turn our attention to applying these lemmas in knot-theoretic contexts.

\section{Knot Types of Horadam Twisted Torus Knots}\label{HTTKs}

We call a twisted torus knot a \textit{Horadam twisted torus knot} when the parameters $\{p,q,r\}$ are consecutive members of a Horadam sequence.
In this section, we consider the knot types of Horadam twisted torus knots with parameters in $(m,n; 1,1)$ Horadam sequences.

There are twelve ways to arrange three consecutive terms of a Horadam sequence as parameters of $K(p,q,r,\pm 1)$, but Lemma \ref{lem:lee} indicates that only six of these types could be distinct knot types. Here, we discuss the six types of twisted torus knots that have Horadam parameters: \begin{enumerate}
\item[]
\item \label{type:bigr} $K(H_{k+2}, H_k, H_{k+1}, -1)$, 
\item \label{type:bigrpos} $K(H_{k+2}, H_k, H_{k+1}, +1)$, 
\item \label{type:bigq} $K(H_{k+2}, H_{k+1}, H_{k}, -1)$,
\item \label{type:bigqpos} $K(H_{k+2}, H_{k+1}, H_{k}, +1)$,
\item \label{type:biggerr}$K(H_{k+1}, H_k, H_{k+2}, -1)$, and
\item \label{type:biggerrpos}$K(H_{k+1}, H_k, H_{k+2}, +1)$.
\end{enumerate} 

In this section, we notice that all of the $(m,n;1,1)$ Horadam twisted torus knots are primitive/primitive. Next, Lemma \ref{lem:positivesarethesame} will tell us that Types \ref{type:bigrpos} and \ref{type:bigqpos} are the same. Further, Lemma \ref{lem:makepthebigone} tells us that the knot type of Type \ref{type:biggerr} is the mirror of Type \ref{type:bigqpos}, and Lemma \ref{lem:type6istype3} tells us that each Type \ref{type:biggerrpos} knot is isotopic to a Type \ref{type:bigq} with shifted index. We will consider Types \ref{type:bigr}, \ref{type:bigq}, and \ref{type:biggerr} in Proposition \ref{prop:knottypes}. 

\begin{lemma}\label{lem:primprim}
All $(m,n;1,1)$-Horadam twisted torus knots are primitive/primitive.
\end{lemma}

\begin{proof}
Dean's Theorem 3.4 \cite{DeanSFSSurgeries} shows that a twisted torus knot of the form $K(p, q, r, \pm1)$ is primitive with respect to $H$ exactly when $r \equiv \pm 1$ or $\pm q \pmod p$ and primitive with respect to $H'$ exactly when $r \equiv \pm 1$ or $\pm p \pmod q$. Since every $(m,n;1,1)$ Horadam twisted torus knot has $r = p \pm q$, we have $r \equiv \pm q \pmod p$ and $r \equiv p \pmod q$. 
\end{proof}

\begin{lemma}\label{lem:positivesarethesame}
Let $p, q \ge 2$ be coprime integers with $p>2q$. Then $K(p,q,p-q,+1)$ is isotopic to $K(p, p-q, q, +1)$.
\end{lemma}

\begin{proof}
Suppose $p, q \ge 2$, $\gcd(p,q) = 1$, and $p >2q$. The knot $K(p,q,p-q,+1)$ can be visualized as the closure of the braid $(\sigma_{p-1} \cdots \sigma_1)^q (\sigma_{p-q-1} \cdots \sigma_1)^{p-q}$. We can rewrite $(\sigma_{p-1} \cdots \sigma_1)^q$ as a full twist on the rightmost $q$ strands followed by those $q$ strands passing over the remaining $p - q$ strands, as shown in Figure \ref{fig:lemma7.1}. As above, the switch indicates that $p$ strands both enter and leave the switch, but that the strands of the braid may change from one grouping to another in the switch. In this figure, we have chosen to think of the full twists on strands as a link surgery diagram for the convenience of visualizing changes to the diagram. 

In this particular diagram, for example, we can see that the link component around the $p-q$ strands at the bottom can be shifted to appear around the $p-q$ strands at the top of the diagram, as in Figure \ref{fig:lemma7.2}. Additionally, in Figure \ref{fig:lemma7.2}, the link component around the rightmost q strands at the top can be shifted to the bottom. From here, we see that because $p-q >q$, the link component surrounding the $q$ strands at the bottom can slide through the switch to create the diagram shown in Figure \ref{fig:lemma7.3}. Now, thinking of the closure of this braid, the $q$ strands on the right at the top and the $q$ strands on the right at the bottom can be pulled out, creating a positive full twist to the rightmost $q$ strands at the bottom, as shown in Figure \ref{fig:lemma7.4} 

From here, we use the fact that a full twist on the rightmost $q$ strands followed by the rightmost $q$ strands passing over the leftmost $p-2q$ strands amounts to $q$ single twists (i.e, $q$ copies of $\sigma_{p-q-1} \cdots \sigma_1$) on all $p-q$ strands. Then Figure \ref{fig:lemma7.4} becomes the closure of $(\sigma_{p-q-1} \cdots \sigma_1)^{p} (\sigma_{q-1} \cdots \sigma_1)^{q}$, which is also the braid whose closure is $K(p-q,p,q,+1)$. By Lemma \ref{lem:lee}, this knot is isotopic to $K(p, p-q, q, +1)$.
\end{proof}

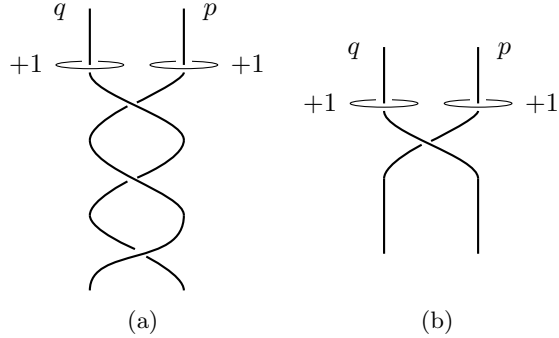
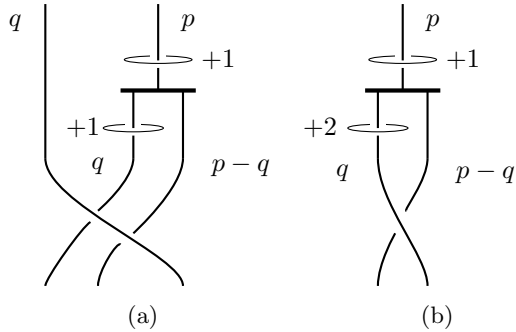
\begin{figure}[ht]
	\begin{subfigure}{.3\textwidth}
		\begin{tikzpicture}
			\draw (2,1.01) ellipse (.45 and 0.05);
			\fill[white] (1.9,1) rectangle (2.1, 1.2);
			\draw[thick] (2,1) -- (2, 1.75); 
			\draw[thick] (1,1.75) -- (1, 1);
			\draw[thick] (1,1) .. controls (1,.6)and (1.4,-.1) .. (1.45,-.1);
			\draw[thick] (2,0.9) .. controls(2,0) and (1,-.5) .. (1,-1);
			\draw[thick] (1.55,-.25) .. controls (1.7,-.5) and (2,-.8) .. (2,-1);
			\draw[ultra thick] (.75,-1) -- (2.25, -1);
			\draw (1,-1.73) ellipse (.45 and 0.05);
			\fill[white] (0.9,-1.72) rectangle (1.1, -1.65);
			\draw[thick] (1,-1) -- (1, -1.72); 
			\draw[thick] (1,-1.85) -- (1, -2.2); 
			\draw[thick] (2,-1) -- (2, -2.2); 
			\node at (0.2,1.7) {$p-q$}; 
			\node at (2.4,1.7) {$q$}; 
			\node at (0.2,-1.25) {$p-q$}; 
			\node at (2.4,-1.25) {$q$}; 
			\node at (2.8,1) {$-1$};  
			\node at (0.15,-1.7) {$-1$};  
		\end{tikzpicture}\caption[b]{}\label{fig:lemma7.1}
	\end{subfigure}
	\begin{subfigure}{.3\textwidth}
		\begin{tikzpicture}
			\draw (2,1.01) ellipse (.45 and 0.05);
			\fill[white] (1.9,1) rectangle (2.1, 1.2);
			\draw[thick] (2,1) -- (2, 1.75); 
			\draw[thick] (3,1.75) -- (3, 1);
			\draw[thick] (2,0.9) .. controls (2,.6)and (2.4,-.1) .. (2.45,-.1);
			\draw[thick] (3,1) .. controls(3,0) and (2,-.5) .. (2,-1);
			\draw[thick] (2.55,-.25) .. controls (2.7,-.5) and (3,-.8) .. (3,-1);
			\draw[ultra thick] (1.75,-1) -- (3.25, -1);
			\draw (3,-1.73) ellipse (.45 and 0.05);
			\fill[white] (2.9,-1.72) rectangle (3.1, -1.65);
			\draw[thick] (3,-1) -- (3, -1.72); 
			\draw[thick] (3,-1.85) -- (3, -2.2); 
			\draw[thick] (2,-1) -- (2, -2.2); 
			\node at (1.2,1.7) {$p-q$}; 
			\node at (3.4,1.7) {$q$}; 
			\node at (1.2,-1.25) {$p-q$}; 
			\node at (3.4,-1.25) {$q$}; 
			\node at (1.2,1) {$-1$};  
			\node at (3.75,-1.7) {$-1$};  
		\end{tikzpicture}\caption{}\label{fig:lemma7.2}
	\end{subfigure}\caption{Braids representing $K(p,q,p-q, +1)$}
\end{figure}
\begin{figure}
	\begin{subfigure}{.3\textwidth}
		\begin{tikzpicture}
			\draw (0,1.01) ellipse (.45 and 0.05);
			\fill[white] (-.1,1) rectangle (.1, 1.2);
			\draw[thick] (1.25,-.3) -- (1.25, 1.75); 
			\draw[thick] (0,1.75) -- (0, 1.);
			\draw[thick] (0, 0.9) -- (0, .6);
			\draw[ultra thick] (-.5,.6) -- (.5, .6);	
			\draw (.33,-.1) ellipse (.4 and 0.05);
			\fill[white] (.23,-.1) rectangle (.43, .2);
			\draw[thick] (-.33, 0.6) -- (-.33, -.3);	
			\draw[thick] (.33, 0.6) -- (.33, -0.1);	
			\draw[thick] (.33,-.2) -- (.33, -.3); 
			\draw[thick] (1.25,-.3) .. controls(1.25,-1.3) and (-.33,-1.2) .. (-.33,-2.3);
			\draw[thick] (-.33,-.3) .. controls (-.33,-.9)and (.3,-1.2) .. (.3,-1.2);
			\draw[thick] (0.33,-.3) .. controls (.33,-.7) and (.8,-.9) .. (.8,-.9);
			\draw[thick] (.95,-1.1) .. controls(1.2,-1.3) and (1.25,-2.3) .. (1.25,-2.3);
			\draw[thick] (.5,-1.4) .. controls (.6,-1.5) and (.66,-2.3) .. (.6,-2.3);
			\node at (-.8,1.7) {$p-q$}; 
			\node at (1.6,1.5) {$q$}; 
			\node at (-.9,.25) {$p-2q$}; 
			\node at (.7,0.25) {$q$}; 
			\node at (-.75,1) {$-1$};  
			\node at (.8,-0.25) {$-1$};  
		\end{tikzpicture}\caption{}\label{fig:lemma7.3}	
		\end{subfigure}
		\begin{subfigure}{.3\textwidth}
		\begin{tikzpicture}
			\draw (1,1.01) ellipse (.45 and 0.05);
			\fill[white] (0.9,1) rectangle (1.1, 1.2);
			\draw[thick] (1,1) -- (1, 1.75); 
			\draw[thick] (1,0.9) -- (1, .5);
			\draw[ultra thick] (.25,.5) -- (1.75, .5);
			\draw (1.35,-.3) ellipse (.4 and 0.05);
			\fill[white] (1.25,-.3) rectangle (1.45, 0.1);
			\draw[thick] (1.35,.5) -- (1.35, -.3); 
			\draw[thick] (1.35,-.4) -- (1.35, -.55); 
			\draw[thick] (.65,0.5) -- (.65, -.55); 
			\draw[thick] (1.35,-.55) .. controls(1.35,-.85) and (.65,-1.3) .. (.65,-1.5);
			\draw[thick] (.65,-.55) .. controls(.65,-.75) and (1,-1) .. (.95,-1);
			\draw[thick] (1.05,-1.1) .. controls(1.05,-1.1) and (1.35,-1.4) .. (1.35,-1.5);
			\draw[thick] (.65, -1.5) -- (.65, -2.25);
			\draw[thick] (1.35, -1.5) -- (1.35, -2.25);
			\node at (0.2,1.5) {$p-q$}; 
			\node at (1.7,.1) {$q$}; 
			\node at (-.1,.1) {$p-2q$}; 
			\node at (1.75,1) {$-1$};  
			\node at (1.9,-.5) {$-2$};  
		\end{tikzpicture}\caption{}\label{fig:lemma7.4}
	\end{subfigure}\caption{Braids representing $K(p,q,p-q, +1)$}

\end{figure}

\begin{lemma}\label{lem:makepthebigone}
Let $p, q\ge 2$ be coprime integers. If $p >q$, $K(p,q,p+q, -1)$ is the mirror image of $K(p, p+q, q, +1)$.
\end{lemma}

\begin{proof}
Suppose $p, q\ge 2$ are coprime integers with $p >q$. Using Figure \ref{fig:r=p+qbraid}, we can think of the knot $K(p,q,p+q, -1)$ as the closure of the braid with one full negative twist on all $p+q$ strands followed by the rightmost $p$ strands passing over the leftmost $q$ strands. This braid has both positive and negative crossings. We use the fact that a full negative twist on all $p+q$ strands can be considered instead as a full twist on the rightmost $p$ strands, a full twist on leftmost $q$ strands and then the rightmost $p$ strands passing under and then over the leftmost $q$ strands. This results in the braid shown in Figure \ref{fig:lemma8.1}. The $p$ strands that pass over the $q$ strands twice can be straightened. Then we move the link components as needed to obtain the diagram in Figure \ref{fig:lemma8.2}. 

Now we see that the link component around the $q$ strands can be shifted to the bottom. Since $p>q$, we can push the link component upward, so that it surrounds the leftmost $q$ strands of the grouping of $p$ strands, as shown in Figure \ref{fig:lemma8.3}. Now we use the braid closure to untwist the leftmost $q$ strands, creating a full negative twist on those strands, to obtain the braid in Figure \ref{fig:lemma8.4}. By converting the negative full twist on $q$ strands followed by the leftmost $q$ strands going over the rightmost $p-q$ strands into $q$ single twists on all $p$ strands, we can see this knot as the closure of the braid $(\sigma_{p-1} \cdots \sigma_1)^{p+q} (\sigma_{p-1} \cdots \sigma_{p-q+1})^q$. By viewing this braid from behind the page, we can see a braid whose closure is the mirror image of $K(p, p+q, q, +1)$.
\end{proof}

\begin{figure}[ht]
	\begin{subfigure}{.3\textwidth}\hspace{1.5in}
		\begin{tikzpicture}
			\draw (2,1) ellipse (.45 and 0.05);
			\draw (3.25,1) ellipse (.45 and 0.05);
			\fill[white] (1.9,1) rectangle (2.1, 1.2);
			\fill[white] (3.15,1) rectangle (3.35, 1.2);
			\draw[thick] (2,1) -- (2, 1.75); 
			\draw[thick] (3.25,1) -- (3.25, 1.75); 
			\draw[thick] (2,.9) .. controls (2,.6)and (3.25,.3) .. (3.25,0);
			\draw[thick] (3.25,0.9) .. controls(3.25,.8) and (2.625,.5) .. (2.625,.5);
			\draw[thick] (2,0) .. controls (2,-.4)and (3.25,-.7) .. (3.25,-1);	
			\draw[thick] (2.5,0.45) .. controls(2.5,.45) and (2,.2) .. (2,0);
			\draw[thick] (2,-2) .. controls (2,-1.4)and (3.25,-1.7) .. (3.25,-1);
			\draw[thick] (3.25,0) .. controls(3.25,-.25) and (2.63,-.5) .. (2.63,-.5);
			\draw[thick] (2.5,-.55) .. controls(2.5,-.55) and (2,-.8) .. (2,-1);
			\draw[thick] (2,-1) .. controls(2,-1.2) and (2.6,-1.45) .. (2.6,-1.45);
			\draw[thick] (2.75,-1.55) .. controls(2.75,-1.55) and (3.25,-1.8) .. (3.25,-2);
			\node at (3.6,1.7) {$p$}; 
			\node at (1.6 ,1.7) {$q$}; 
			\node at (4.1,1) {$+1$};  
			\node at (1.15,1) {$+1$};  
		\end{tikzpicture}\caption[b]{}\label{fig:lemma8.1}
	\end{subfigure}
	\begin{subfigure}{.3\textwidth}
		\begin{tikzpicture}
			\draw (2,1) ellipse (.45 and 0.05);
			\draw (3.25,1) ellipse (.45 and 0.05);
			\fill[white] (1.9,1) rectangle (2.1, 1.2);
			\fill[white] (3.15,1) rectangle (3.35, 1.2);
			\draw[thick] (2,1) -- (2, 1.75); 
			\draw[thick] (3.25,1) -- (3.25, 1.75); 
			\draw[thick] (2,.9) .. controls (2,.6)and (3.25,.3) .. (3.25,0);
			\draw[thick] (3.25,0.9) .. controls(3.25,.8) and (2.625,.5) .. (2.625,.5);
			\draw[thick] (2.5,0.45) .. controls(2.5,.45) and (2,.2) .. (2,0);
			\draw[thick] (2,0) -- (2, -1); 
			\draw[thick] (3.25,0) -- (3.25, -1.); 
			\node at (3.6,1.7) {$p$}; 
			\node at (1.6 ,1.7) {$q$}; 
			\node at (4.1,1) {$+1$};  
			\node at (1.15,1) {$+1$};  
			\fill[white] (2,-1) rectangle (3, -1.5);
		\end{tikzpicture}\caption{}\label{fig:lemma8.2}
	\end{subfigure}\caption{Braids representing $K(p,q,p-q, +1)$}
\end{figure}

\begin{figure}
	\begin{subfigure}{.3\textwidth}
		\begin{tikzpicture}
			\draw (0,1.01) ellipse (.45 and 0.05);
			\fill[white] (-.1,1) rectangle (.1, 1.2);
			\draw[thick] (-1.5,-.3) -- (-1.5, 1.75); 
			\draw[thick] (0,1.75) -- (0, 1.);
			\draw[thick] (0, 0.9) -- (0, .6);
			\draw[ultra thick] (-.5,.6) -- (.5, .6);	
			\draw (-.33,.1) ellipse (.4 and 0.05);
			\fill[white] (-.23,.1) rectangle (-.43, .2);
			\draw[thick] (.33, 0.6) -- (.33, -.3);	
			\draw[thick] (-.33, 0.6) -- (-.33, 0.1);	
			\draw[thick] (-.33,0) -- (-.33, -.3); 
			\draw[thick] (-1.5,-.3) .. controls(-1.5,-.9) and (.33,-1.6) .. (.33,-2);
			\draw[thick] (-.33,-.3) .. controls (-.33,-.6)and (-.8,-1.) .. (-.8,-1.);
			\draw[thick] (0.33,-.3) .. controls (.33,-.7) and (-.33,-1.3) .. (-.33,-1.3);
			\draw[thick] (-.9,-1.15) .. controls(-1.,-1.2) and (-1.5,-1.9) .. (-1.5,-2);
			\draw[thick] (-.5,-1.45) .. controls (-.54,-1.4) and (-.8,-1.8) .. (-.8,-2);
			\node at (1.1,-.4) {$p-q$}; 
			\node at (-.8,-.4) {$q$}; 
			\node at (.4,1.5) {$p$};
			\node at (-1.9,1.5) {$q$}; 
			\node at (-1., .1) {$+1$};  
			\node at (.8,1.01) {$+1$};  
		\end{tikzpicture}\caption{}\label{fig:lemma8.3}	
		\end{subfigure}
		\begin{subfigure}{.3\textwidth}
		\begin{tikzpicture}
			\draw (0,1.01) ellipse (.45 and 0.05);
			\fill[white] (-.1,1) rectangle (.1, 1.2);
			\draw[thick] (0,1.75) -- (0, 1.);
			\draw[thick] (0, 0.9) -- (0, .6);
			\draw[ultra thick] (-.5,.6) -- (.5, .6);	
			\draw (-.33,.1) ellipse (.4 and 0.05);
			\fill[white] (-.23,.1) rectangle (-.43, .2);
			\draw[thick] (.33, 0.6) -- (.33, -.3);	
			\draw[thick] (-.33, 0.6) -- (-.33, 0.1);	
			\draw[thick] (-.33,0) -- (-.33, -.3); 
			\draw[thick] (-.33,-.3) .. controls(-.33,-.9) and (.33,-1.6) .. (.33,-2);
			\draw[thick] (.33,-.3) .. controls (.33,-.6)and (0,-1.) .. (0.05,-1.);
			\draw[thick] (-.1,-1.3) .. controls(-.1,-1.3) and (-.33,-1.7) .. (-.33,-2);
			\node at (1.1,-.5) {$p-q$}; 
			\node at (-.8,-.5) {$q$}; 
			\node at (.4,1.5) {$p$};
			\node at (-1.1, .1) {$+2$};  
			\node at (.8,1.01) {$+1$};  
		\end{tikzpicture}\caption{}\label{fig:lemma8.4}
	\end{subfigure}\caption{Braids representing $K(p,q,p-q, +1)$}
\end{figure}
\vspace{.3in}

\begin{lemma}\label{lem:type6istype3}
Let $\mathcal H$ be the $(m,n)$-Horadam sequence where $m$ and $n$ are positive integers and $\gcd(m,n) = 1$. The twisted torus knots $K(H_{k+3}, H_{k+2}, H_{k+1}, -1)$ and $K(H_{k+1}, H_k, H_{k+2}, +1)$ are isotopic. 
\end{lemma}

\begin{proof}
Suppose $\mathcal H$ is the $(m,n)$-Horadam sequence where $m$ and $n$ are positive integers and $\gcd(m,n) = 1$. The twisted torus knot $K(H_{k+3}, H_{k+2}, H_{k+1}, -1)$ can be considered to be the closure of the braid $(\sigma_{H_{k+3} - 1} \cdots \sigma_1)^{H_{k+2}}(\sigma_{H_{k+1}-1} \cdots \sigma_1)^{-H_{k+1}}$. Since $H_{k+3} = H_{k+2}+ H_{k+1}$, we see this braid as in Figure \ref{fig:lemma9.1}, where the top portion of the braid, $(\sigma_{H_{k+3} - 1} \cdots \sigma_1)^{H_{k+2}}$, is considered to be a full twist on the rightmost $H_{k+2}$ strands followed by those strands passing over the leftmost $H_{k+1}$ strands.

Because $H_{k+1} < H_{k+2}$, we can slide the link component at the bottom of Figure \ref{fig:lemma9.1} upward to surround the leftmost $H_{k+1}$ strands of the grouping of $H_{k+2}$ strands on the right. From this, we obtain the braid in Figure \ref{fig:lemma9.2}. Now, because we are working with braid closures, we untwist the leftmost $H_{k+1}$ strands to add a full positive twist. Now we have $H_{k+1}$ strands with a full negative twist followed by a full positive twist. These full twists cancel each other, and we are left with the braid shown in Figure \ref{fig:lemma9.3}, which is the braid representing $K(H_{k+1}, H_k, H_{k+2}, +1)$.
\end{proof}

\begin{figure}[ht]
	\begin{subfigure}{.3\textwidth}
		\begin{tikzpicture}
			\draw (2,1.01) ellipse (.45 and 0.05);
			\fill[white] (1.9,1) rectangle (2.1, 1.2);
			\draw[thick] (2,1) -- (2, 1.75); 
			\draw[thick] (1,1.75) -- (1, 1);
			\draw[thick] (1,1) .. controls (1,.6)and (1.4,-.1) .. (1.45,-.1);
			\draw[thick] (2,0.9) .. controls(2,0) and (1,-.5) .. (1,-1);
			\draw[thick] (1.55,-.25) .. controls (1.7,-.5) and (2,-.8) .. (2,-1);
			\draw[ultra thick] (.75,-1) -- (2.25, -1);
			\draw (1,-1.73) ellipse (.45 and 0.05);
			\fill[white] (0.9,-1.72) rectangle (1.1, -1.65);
			\draw[thick] (1,-1) -- (1, -1.72); 
			\draw[thick] (1,-1.85) -- (1, -2.2); 
			\draw[thick] (2,-1) -- (2, -2.2); 
			\node at (0.3,1.7) {$H_{k+1}$}; 
			\node at (2.6,1.7) {$H_{k+2}$}; 
			\node at (0.2,-1.25) {$H_{k+1}$}; 
			\node at (2.6,-1.25) {$H_{k+2}$}; 
			\node at (2.8,1) {$-1$};  
			\node at (0.15,-1.7) {$+1$};  
		\end{tikzpicture}\caption[b]{}\label{fig:lemma9.1}
	\end{subfigure}
	\begin{subfigure}{.3\textwidth}
		\begin{tikzpicture}
			\draw (0,1.01) ellipse (.45 and 0.05);
			\fill[white] (-.1,1) rectangle (.1, 1.2);
			\draw[thick] (-1.5,-.3) -- (-1.5, 1.75); 
			\draw[thick] (0,1.75) -- (0, 1.);
			\draw[thick] (0, 0.9) -- (0, .6);
			\draw[ultra thick] (-.5,.6) -- (.5, .6);	
			\draw (-.33,.1) ellipse (.4 and 0.05);
			\fill[white] (-.23,.1) rectangle (-.43, .2);
			\draw[thick] (.33, 0.6) -- (.33, -.3);	
			\draw[thick] (-.33, 0.6) -- (-.33, 0.1);	
			\draw[thick] (-.33,0) -- (-.33, -.3); 
			\draw[thick] (-1.5,-.3) .. controls(-1.5,-.9) and (.33,-1.6) .. (.33,-2);
			\fill[white] (-.8, -1.05) circle (3pt);
			\fill[white] (-.4, -1.35) circle (3pt);
			\draw[thick] (-.33,-.3) .. controls (-.33,-.8)and (-1.5,-1.5) .. (-1.5,-2);
			\draw[thick] (0.33,-.3) .. controls (.33,-.8) and (-.8,-1.5) .. (-.8,-2);
			\node at (.8,-.4) {$H_{k}$}; 
			\node at (-.8,-.4) {$H_{k+1}$}; 
			\node at (.6,1.5) {$H_{k+2}$};
			\node at (-2,1.5) {$H_{k+1}$}; 
			\node at (-1., .1) {$+1$};  
			\node at (.8,1.01) {$-1$};  
		\end{tikzpicture}\caption{}\label{fig:lemma9.2}
	\end{subfigure}\caption{Braids representing $K(H_{k+3}, H_{k+2}, H_{k+1}, -1)$}
	\end{figure}
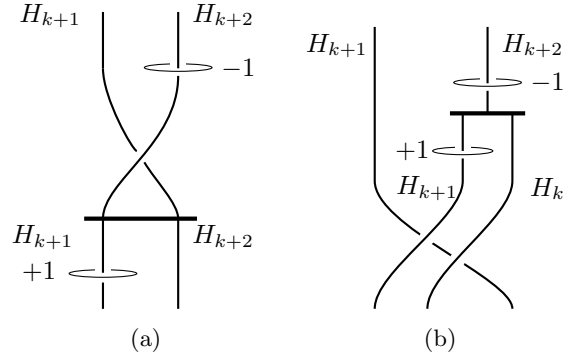
	
	\begin{figure}
		\begin{tikzpicture}
			\draw (0,0.41) ellipse (.45 and 0.05);
			\fill[white] (-.1,.4) rectangle (.1, .6);
			\draw[thick] (0,1.15) -- (0, .4);
			\draw[thick] (0, 0.3) -- (0, 0);
			\draw[ultra thick] (-.5,0) -- (.5, 0);	
			\draw[thick] (.33, 0) -- (.33, -0.3);	
			\draw[thick] (-.33,0) -- (-.33, -.3); 
			\draw[thick] (-.33,-.3) .. controls(-.33,-.9) and (.33,-1.6) .. (.33,-2);
			\fill[white] (0, -1.25) circle (3pt);
			\draw[thick] (.33,-.3) .. controls (.33,-.9)and (-.33,-1.7) .. (-.33,-2);
			\node at (.8,-.5) {$H_{k}$}; 
			\node at (-.9,-.5) {$H_{k+1}$}; 
			\node at (.6,.9) {$H_{k+2}$};
			\node at (.8,.41) {$-1$};  
			\fill[white] (-1,.1) rectangle (-1.9, .2);
		\end{tikzpicture}\caption{A braid representing $K(H_{k+3}, H_{k+2}, H_{k+1}, -1)$ }\label{fig:lemma9.3}

\end{figure}

Next, we state two theorems of Lee that will serve as an important part of the structure of the proof of Proposition \ref{prop:knottypes}.

\begin{theorem}[Lee \cite{LeeTTKTorusPos}]\label{thm:ttktoruspos}
Let $p$, $q$, and $s$ be positive integers such that $p$ and $q$ are coprime and $2\le q <p$. Let $r$ be an integer with $2 \le r \le p + q$, $r\neq p$, and $r$ is not a multiple of $q$. Then the twisted torus knot $K(p,q,r, s)$ is a torus knot if and only if $(p,q,r,s) = (ab + 1,b, b-1, 1)$ for some integers $a \ge 1$ and $b \ge 3$.
\end{theorem}

\begin{theorem}[Lee \cite{LeeTTKTorus}]\label{thm:ttktorusqbig}
Let $p$, $q$, $k$ be positive integers such that $p$ and $q$ are coprime, $2\le q <p$ and $2 \le p - kq$. Suppose $p -kq < q$. Then the twisted torus knot $K(p,q,p-kq, -1)$ is a torus knot if and only if $(p,q,p-kq) = ((a+1)b-1, b, b-1)$ for some integers $a \ge 1$ and $b \ge 3$.
\end{theorem}

\vspace{.1in}

Now that we have established some structure, we prove that for a fixed $m$ and $n$, all of the Type \ref{type:bigr} knots are the same, up to mirror image, but for the Type \ref{type:bigq} and Type \ref{type:biggerr} knots, every non-torus knot is different. 
\begin{proposition}\label{prop:knottypes}
Let $\mathcal H$ be the $(m,n)$-Horadam sequence where $m$ and $n$ are positive integers and $\gcd(m,n) = 1$.
\begin{enumerate}
\item\label{part:type1} For all integers $k\ge 1$, the knot type of $K(H_{k+2}, H_{k}, H_{k+1}, -1)$ is the same as $K(H_2, H_0, H_1, -1)$ or its mirror image.
\item\label{part:type3}  Let $K_k$ denote the twisted torus knot $K(H_{k+2}, H_{k+1}, H_{k}, -1)$. If $m, n \ge 2$, then $K_i$ and $K_j$ have different knot types for all nonnegative integers $i$ and $j$ with $i \neq j$. 
\item\label{part:type5} Let $K'_k$ denote the twisted torus knot $K(H_{k+1}, H_{k}, H_{k+2}, -1)$. If $m, n\ge 2$, then $K'_i$ and $K'_j$ have different knot types for all nonnegative integers $i$ and $j$ with $i \neq j$. 
\end{enumerate}
\end{proposition}

When $m=0$ and $n=1$, $\mathcal H$ is the Fibonacci sequence, so Proposition \ref{prop:knottypes}.(\ref{part:type1}) generalizes Lee's result \cite{LeeTTKUnknots} that twisted torus knots of the form $K(F_{n+2}, F_n, F_{n+1}, -1)$ are all the unknot.  
Proposition \ref{prop:knottypes}.(\ref{part:type3}) tells us that each of the non-torus knots in this family is distinct, providing contrast to Lee's \cite{LeeTTKUnknots} family of twisted torus unknots with parameters in the Fibonacci sequence.

\begin{proof}
We begin with parts \ref{part:type3} and \ref{part:type5}, first establishing that for $m,n \ge 2$, none of these knots are torus knots. For part  \ref{part:type3}, because $H_{k+1} - H_k \ge 2$, Theorem \ref{thm:ttktoruspos} tells us that $K_k$ is not a torus knot. For part \ref{part:type5}, we use Lemma \ref{lem:makepthebigone} and then Lemma \ref{lem:lee} to say $K_k' = K(H_{k+1}, H_{k}, H_{k+2}, -1)$ is isotopic to the mirror image of $K(H_{k+2}, H_{k+1}, H_{k}, +1)$, which is not a torus knot by Theorem \ref{thm:ttktorusqbig} since $H_{k+1} - H_k \ge 2$.

From here, the proofs of parts \ref{part:type3} and \ref{part:type5} will proceed similarly, so we present the proof of part \ref{part:type3}, noting that the surface slopes for $K_k$ and $K'_k$ are $s_{k+1}$ and $-t_k$, respectively.  Consider $K_{i}$ and $K_{j}$ with $i \neq j$. The surface slopes for $K_{i}$ and $K_{j}$ are  $s_{i+1}$ and $s_{j+1}$, respectively, and Lemma \ref{lem:slopesaredifferent} tells us that $|s_{i+1} - s_{j+1}| \ge 2$. 

By Lemma \ref{lem:primprim}, we have that $K_i$ and $K_j$ are primitive/primitive curves, so surgery at the surface slope of each is a lens space. The Cyclic Surgery Theorem \cite{CGLS} tells us that for knots that are not torus knots, the surgery slopes of two lens space surgeries on the same knot must be consecutive integers. Since $|s_{i+1} - s_{j+1}| \ge 2$, $K_i$ and $K_j$ cannot have slopes that are consecutive integers, so they cannot be the same knot.

To prove part \ref{part:type1}, we use induction, noting that both the base case and the inductive case are given by the braids shown in Figures \ref{fig:prop1.1} and \ref{fig:prop1.2}. In the base case, we replace the $k$ in figures with a $0$. We begin by noting that $K(H_{k+2}, H_{k}, H_{k+1}, -1)$ can be drawn as the closure of the braid in Figure \ref{fig:prop1.1}. Since $H_{k+1}$ is larger than $H_k$, we can slide the link component around the $H_{k+1}$ strands to the top of the braid. Then we slide the link component around the $H_k$ strands to the bottom, where it moves through the switch, to produce the braid shown in Figure \ref{fig:prop1.2}. Now, we can untwist the rightmost $H_k$ strands to decrease the number of strands by $H_k$ and obtain another full positive twist on the rightmost $H_k$ strands. This produces a link where the surgery coefficient on the rightmost $H_k$ strands is $-2$. Now we use the fact that a full positive twist on the rightmost $H_k$ strands followed by those $H_k$ strands passing over the leftmost $H_{k-1}$ strands can be written as $(\sigma_{H_{k+1}-1} \cdots \sigma_1)^{H_k}$. Now we have that Figure \ref{fig:prop1.2} can be written as the braid $(\sigma_{H_{k+1}-1} \cdots \sigma_1)^{-H_{k+1}}(\sigma_{H_{k+1}-1} \cdots \sigma_1)^{H_k} (\sigma_{H_{k}-1} \cdots \sigma_1)^{H_k}$, which is the same as $(\sigma_{H_{k+1}-1} \cdots \sigma_1)^{H_k-H_{k+1}} (\sigma_{H_{k}-1} \cdots \sigma_1)^{H_k}$. Because $H_k-H_{k+1} = -H_{k-1}$, the braid in question is on $H_{k+1}$ strands with $H_{k-1}$ single negative twists, followed by a full positive twist on the first $H_k$ strands. The closure of this braid is the mirror image of $K(H_{k+1}, H_{k-1}, H_k, -1)$. \end{proof}

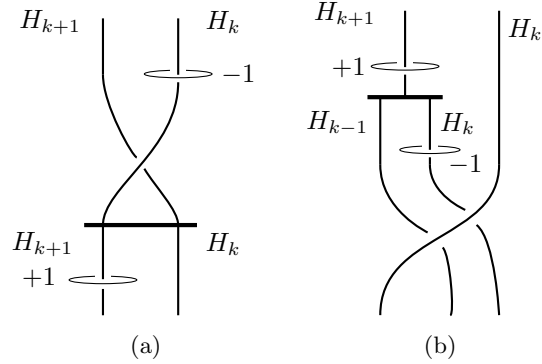
\begin{figure}[ht]
	\begin{subfigure}{.3\textwidth}
		\begin{tikzpicture}
			\draw (2,1.01) ellipse (.45 and 0.05);
			\fill[white] (1.9,1) rectangle (2.1, 1.2);
			\draw[thick] (2,1) -- (2, 1.75); 
			\draw[thick] (1,1.75) -- (1, 1);
			\draw[thick] (1,1) .. controls (1,.6)and (1.4,-.1) .. (1.45,-.1);
			\draw[thick] (2,0.9) .. controls(2,0) and (1,-.5) .. (1,-1);
			\draw[thick] (1.55,-.25) .. controls (1.7,-.5) and (2,-.8) .. (2,-1);
			\draw[ultra thick] (.75,-1) -- (2.25, -1);
			\draw (1,-1.73) ellipse (.45 and 0.05);
			\fill[white] (0.9,-1.72) rectangle (1.1, -1.65);
			\draw[thick] (1,-1) -- (1, -1.72); 
			\draw[thick] (1,-1.85) -- (1, -2.2); 
			\draw[thick] (2,-1) -- (2, -2.2); 
			\node at (0.3,1.7) {$H_{k+1}$}; 
			\node at (2.6,1.7) {$H_{k}$}; 
			\node at (0.2,-1.25) {$H_{k+1}$}; 
			\node at (2.6,-1.25) {$H_{k}$}; 
			\node at (2.8,1) {$-1$};  
			\node at (0.15,-1.7) {$+1$};  
		\end{tikzpicture}\caption[b]{}\label{fig:prop1.1}
	\end{subfigure}
	\begin{subfigure}{.3\textwidth}
		\begin{tikzpicture}
			\draw (0,1.01) ellipse (.45 and 0.05);
			\fill[white] (-.1,1) rectangle (.1, 1.2);
			\draw[thick] (1.25,-.3) -- (1.25, 1.75); 
			\draw[thick] (0,1.75) -- (0, 1.);
			\draw[thick] (0, 0.9) -- (0, .6);
			\draw[ultra thick] (-.5,.6) -- (.5, .6);	
			\draw (.33,-.1) ellipse (.4 and 0.05);
			\fill[white] (.23,-.1) rectangle (.43, .2);
			\draw[thick] (-.33, 0.6) -- (-.33, -.3);	
			\draw[thick] (.33, 0.6) -- (.33, -0.1);	
			\draw[thick] (.33,-.2) -- (.33, -.3); 
			\draw[thick] (1.25,-.3) .. controls(1.25,-1.3) and (-.33,-1.2) .. (-.33,-2.3);
			\draw[thick] (-.33,-.3) .. controls (-.33,-.9)and (.3,-1.2) .. (.3,-1.2);
			\draw[thick] (0.33,-.3) .. controls (.33,-.7) and (.8,-.9) .. (.8,-.9);
			\draw[thick] (.95,-1.1) .. controls(1.2,-1.3) and (1.25,-2.3) .. (1.25,-2.3);
			\draw[thick] (.5,-1.4) .. controls (.6,-1.5) and (.66,-2.3) .. (.6,-2.3);
			\node at (-.8,1.7) {$H_{k+1}$}; 
			\node at (1.6,1.5) {$H_{k}$}; 
			\node at (-.9,.25) {$H_{k-1}$}; 
			\node at (.7,0.25) {$H_{k}$}; 
			\node at (-.75,1) {$+1$};  
			\node at (.8,-0.3) {$-1$};  
		\end{tikzpicture}\caption[b]{}\label{fig:prop1.2}
	\end{subfigure}\caption{Braids representing $K(H_{k+2}, H_{k}, H_{k+1}, -1)$}
\end{figure}

\begin{theorem}[Lee \cite{LeeTTKTorus}]\label{thm:ttktorusqsmall}
Let $p$ and $q$ be coprime integers with $p\ge 3$, and $q \ge 2$ so that $q < p - q$. Then the twisted torus knot $K(p,q,p-q, -1)$ is a torus knot if and only if $(p,q,p-q)$ has one of the following forms:
\begin{enumerate}
\item $(aF_{b+3} - F_{b+2}, aF_{b+1}-F_b, aF_{b+2} - F_{b+1})$ for some integers $a \ge 2$ and $b \ge 1$ such that $(a, b) \neq (2, 1)$, or
\item $(aF_{b+1} + F_{b+2}, aF_{b-1} + F_b, a F_b + F_{b+1})$ for some integers $a, b \ge 2$
\end{enumerate}
where $F_b$ is the $b$th Fibonacci number. When $(p,q,p-q)$ is in one of these forms, $K(p,q,p-q,-1) = T(a+1, (-1)^b a)$.
\end{theorem}

As corollary to Theorem \ref{thm:ttktorusqsmall}, we determine when the knots in Proposition \ref{prop:knottypes}(\ref{part:type1}) are torus knots.

\begin{cor}
The twisted torus knot $K(H_{k+2}, H_{k}, H_{k+1}, -1)$ for $H_i \in \mathcal H_{m,n}$ is a torus knot if and only if $(m,n)$ is a maximal pair.
\end{cor}

\begin{proof}
Theorem \ref{thm:ttktorusqsmall} tells us that $K(H_{k+2}, H_{k}, H_{k+1}, -1)$ for $H_i \in \mathcal H_{m,n}$ is a torus knot if and only if $(H_{k+2}, H_{k}, H_{k+1})$ is of one of the forms $(aF_{b+3} - F_{b+2}, aF_{b+1}-F_b, aF_{b+2} - F_{b+1})$, with $a \ge 2$, $b \ge 1$ and $(a,b) \neq (2,1)$, or $(aF_{b+1} + F_{b+2}, aF_{b-1} + F_b, a F_b + F_{b+1})$, with  $a, b \ge 2$. 

We have $(H_{k+2}, H_{k}, H_{k+1})$ of the form $(aF_{b+3} - F_{b+2}, aF_{b+1}-F_b, aF_{b+2} - F_{b+1})$ if and only if
$H_k = G_b$ and $H_{k+1} = G_{b+1}$ for some $b \ge 1$ where $G_b, G_{b+1} \in \mathcal H_{-1, a}$. Since $G_b, G_{b+1} \in \mathcal H_{-1, a}$, we have $\mathcal H_{G_b, G_{b+1}}$ is a subsequence of $\mathcal H_{-1, a}$, and Lemma \ref{lem:subseq} tells us that this is possible if and only if $(G_b, G_{b+1})$ is a maximal pair. Further because $(G_b, G_{b+1}) = (H_k,H_{k+1})$,  $(G_b, G_{b+1})$ is a maximal pair if and only if  $(H_k,H_{k+1})$ is a maximal pair. Finally, because $H_k, H_{k+1} \in \mathcal H_{m,n}$, Lemma \ref{lem:maxpairs} tells us that $(H_k,H_{k+1})$ is a maximal pair if and only if $(m,n)$ is a maximal pair.

If $(H_{k+2}, H_{k}, H_{k+1})$ is of the form $(aF_{b+1} + F_{b+2}, aF_{b-1} + F_b, a F_b + F_{b+1})$ with $a,b \ge 2$, the proof is similar, except that we must first note that  the $b$th term in $ \mathcal H_{a,1}$ is the same as $G_{b-1} \in \mathcal H_{1, a+1}$. One can see this by using the definition of the Fibonacci sequence, along with Lemma \ref{lem:seq}: $aF_{b-1} + F_b =  aF_{b-1} + (F_{b-1} + F_{b-2}) =  F_{b-2} +(a+1)F_{b-1} $.

\end{proof}

Parts \ref{part:type3} and  \ref{part:type5} of Proposition \ref{prop:knottypes} is corroborated by Kadokami's work \cite{Kadokami}, as these knots can be drawn as braids in a way that shows they are equivalent to $b^-(H_{k-1}, H_{k})$ and $b^+(H_{k}, H_{k+1})$, respectively. Kadokami does not reference these braids as twisted torus knots, however, so the characterization as twisted torus knots is novel. 

\section{Primitive/primitive and primitive/Seifert twisted torus knots}\label{primplus}

As shown in Lemma \ref{lem:primprim}, all of the $(m,n; 1,1)$-Horadam twisted torus knots are primitive/primitive. In this section, we provide a list of all of the primitive/primitive twisted torus knots and note that not all of the knots in this list can be seen immediately as Horadam twisted torus knots. Further, we list all of the twisted torus knots that are middle- or hyper-Seifert with respect to $H$ and primitive with respect to $H'$, expanding and correcting work of Dean's that listed the twisted torus knots that are middle-Seifert with respect to $H$ and primitive with respect to $H'$.

\begin{theorem}\label{thm:pp}
Let $p$ and $q$ be positive, coprime integers which are each at least 2 and let $r$ be an integer with $2 \le r \le p+q$. The primitive/primitive twisted torus knots are those of the form $K(p,q,r, \pm 1)$ where $(p,q,r)$ is a triple of the following type:
\begin{enumerate}
\item\label{list:1} $(p,q,p+q)$,
\item\label{list:2} $(p,q,p-q)$,
\item\label{list:3} $(2ij+i + j + \frac{1+\delta}{2}, 2j+1, 2ij + i + j + \frac{1- \delta}{2})$, where $\delta = \pm 1$, $i \ge 0$, $j \ge 1$,
\item\label{list:4} $(3j+1  + \frac{1+\varepsilon}{2}, 2j+1, 4j + 2 + \varepsilon)$, where $\varepsilon = \pm 1$ and $j \ge 1$, or
\item\label{list:5} $(2jk + k + 2 \eps, 2j+1, 2jk+k + \eps)$ where $\eps = \pm 1$, $j \ge 1$, $k \ge 1$, and $(j,k, \eps) \neq (1, 1, -1)$.
\end{enumerate}
\end{theorem}

\begin{proof}
Dean \cite{DeanSFSSurgeries} shows that a twisted torus knot $K(p,q,r, \pm 1)$ is primitive/primitive if and only if $r \equiv \pm 1$ or $\pm q \pmod p $ and $r \equiv \pm 1$ or $\pm p \pmod q$. We have four cases to consider.
\begin{enumerate}[(a)]
\item\label{case:D} $r \equiv \pm q \pmod p$ and $r \equiv \pm p \pmod q$
\item\label{case:A} $r \equiv \pm 1 \pmod p$ and $r \equiv \pm p \pmod q$
\item\label{case:B} $r \equiv \pm q \pmod p$ and $r \equiv \pm 1 \pmod q$
\item\label{case:C} $r \equiv \pm 1 \pmod p$ and $r \equiv \pm 1 \pmod q$
\end{enumerate}

Here, we assume that $p>q$, as Lemma \ref{lem:lee} tells us that $K(p,q,r,\pm 1)$ and $K(q,p,r,\pm1)$ are the same knot. We can recover the case where $q>p$.

In Case \ref{case:D}, we have $r = kq + \eps p = lp + \delta q$, where $\eps = \pm 1$, $\delta = \pm 1$, and $k$ and $l$ are nonzero integers. (Neither $k$ nor $l$ can be zero because $p$ and $q$ are relatively prime and $p,q \ge 2$.) Because $ r \le p + q$, we have the requirement $ lp + \delta q \le p+q$. Because $p>q$, $lp + \delta q > (l-1) p$. This means $(l-1)p < p+q$. Since $p+ q < 2p$, we have $l < 3$. On the other hand, $r$ is positive,  so $p>q$ tells us that $l \ge 1$. Then $l \in \{1, 2\}$.

When $l = 1$, $r = p + \delta q = kq + \eps p$, so $(1-\eps) p = (k - \delta) q$. If $\eps = 1$,  then $k = \delta$ and $p$ and $q$ are independent of each other. This gives us the set of triples $(p, q, p + \delta q)$, which are Types \ref{list:1} and \ref{list:2} from the statement of the theorem. If $\eps = -1$, then $q = 2$ and $p = k - \delta$. In this case $r = k + \delta$, so we have the triple $(k- \delta, 2, k + \delta)$, which is a special case of Type \ref{list:1} or \ref{list:2}, depending on the value of $\delta$.

When $l = 2$, $r = 2 p + \delta q = kq  + \eps p$. The requirement that $r \le p + q$ forces $\delta = -1$, so we have $2p - q = k q + \eps p$. Then $(2- \eps) p = (k+1) q$. Since $2- \eps \in \{1, 3\}$ and $q \ge 2$, we have $\eps = -1$, $q = 3$, $p= k+1$ and $r = 2k - 1$. Revisiting the requirement that $r \le p + q$, we have $2k-1 \le k+4$, so $k \le 5$. Since $p > q$, we also have $k \ge 3$. This provides the three triples $(4,3,5)$, $(5, 3, 7)$, and $(6, 3, 9)$, but the triple $(6,3,9)$ has $\gcd(p,q) = 3$. The triples $(4,3,5)$ and $(5, 3, 7)$ are special cases of Type \ref{list:4} with $j=1$. 

In Case \ref{case:A}, $r = kq + \eps p  = lp + \delta$, where $\eps = \pm 1$, $\delta = \pm 1$, and $k$ and $l$ are integers. The requirement that $r \ge 2$ provides the restriction that $l$ be a positive integer. Since $r \le p+q$, we have $l \le 2$, so that $l \in \{1,2\}$.

When $l = 1$, we have $kq + \eps p = p + \delta$, so $(1-\eps) p = kq - \delta$. This forces $\eps = -1$ so that $2p = kq - \delta$ and $k \ge 1$. Since $2p$ is even, we have that  $k$ and $q$ are both odd, so we let $k = 2i+1$ for some nonnegative integer $i$ and $q = 2j+1$ for some positive integer $j$. Since $2p = kq - \delta$, we have $p = 2ij + i + j + \frac{1-\delta}{2}$ and $r = p + \delta = 2ij + i + j + \frac{1+\delta}{2}$. This is Type \ref{list:3}.

When $l = 2$, $kq + \eps p = 2p + \delta$. We have $r = 2p + \delta \le p+ q$, so $p \le q - \delta$. Since $p>q$, this means $\delta = -1$, $p = q+1$, and $r = 2q+1$. This produces the triple $(q+1, q, 2q+1)$, a special case of Type \ref{list:1}. 

In Case \ref{case:B}, $r = kq + \eps = lp + \delta q$, where $\eps = \pm 1$, $\delta = \pm 1$, and $k$ and $l$ are integers. The requirement that $r \ge 2$ provides the restriction that $l$ be a nonnegative integer and $k \ge 1$. Since $r \le p+q$, we have $l \le 2$, so that $l \in \{0, 1,2\}$.

When $l = 0$, the two expressions for $r$ reveal that $kq + \eps = \delta q$, which is impossible, as it would require $q$ to be 1, contradicting our assumptions. When $l = 1$,  $kq  + \eps = p + \delta q$. Then we have $p = (k - \delta ) q + \eps$ and $r = kq  + \eps$. The triple $((k - \delta ) q + \eps, q, kq + \eps)$ is a special case of Type \ref{list:1} or \ref{list:2}, depending on the value of $\delta$.

When $l = 2$, $r = kq + \eps = 2p + \delta q$. Since $r \le p+q$, we have $\delta = -1$. Using the two expressions for $r$, we have $2p = (k+1)q  + \eps$. Then we know that $k+1$ and $q$ are both odd, so we write $k = 2i$ and $q = 2j + 1$ for some positive integers $i$ and $j$. Then $p = 2ij + i + j + \frac{1+\eps}{2}$ and $r = 4ij + 2i +\eps$. Since $r \le p + q$, we need $4ij + 2i +\eps \le 2ij + i + 3j + \frac{1+\eps}{2} +1$, which can be rearranged as $i(2j + 1)  \le 3j+1 + \frac{1-\eps}{2}$. If $i \ge 2$, $i(2j+1) \ge 4j+4$, which is strictly greater than $3j+1 + \frac{1-\eps}{2}$ for all values of $j$. Then $i = 1$, and we have $p = 3j+1 + \frac{1+\eps}{2}$ and $r = 4j + 2 + \eps$. These are the triples that make up Type \ref{list:4}.

In Case \ref{case:C}, $r = kq + \eps = lp + \delta $, where $\eps = \pm 1$, $\delta = \pm 1$, and $k$ and $l$ are integers. The requirement that $r \ge 2$ provides the restriction that both $k$ and $l$ are positive. Since $r \le p+q$, we have $l \le 2$, so that $l \in \{1,2\}$. 

When $\eps = \delta$, we have $kq = lp$. Since $p$ and $q$ are relatively prime, this means $ l = \mu q$ and $k = \mu p$ for some positive integer $\mu$. Then $r = \mu p q + \eps$. Because $q \ge 2$, if $\mu \ge 2$ or $q \ge 3$, then $r \ge 3p + \eps > p+ q$. Then $\mu = 1$ and $q = 2$. The requirement that $r \le p+q$ can only hold then when  $2p+ \eps \le p + 2$, which means $\eps = -1$ and $p = 3$. Thus, we have the triple $(3, 2, 5)$, which is a special case of Type \ref{list:1}.

Now suppose $\eps = - \delta$, so $kq + \eps = lp - \eps$. When $l = 1$, $p= kq + 2\eps$, and $r = kq + \eps$. Because $p$ has the same parity as $kq$, the requirement that $p$ and $q$ are relatively prime forces $q$ to be odd. Then we can write $q = 2j + 1$ for some positive integer $j$, and we have the examples of Type \ref{list:5}. 

When $l = 2$, $r = 2p - \eps$. Since $r \le p+q$ and $p >q$, we must have that $\eps = 1$. Since $2p -1 \le p+q$, we have $p \le q+1$. That is, $p = q+1$ and $r = 2q+1$, so we have the triple $(q+1, q, 2q+1)$, which is a special case of Type \ref{list:1}.

This completes the proof.
\end{proof}

\vspace{.1in}

 The Horadam twisted torus knots can be described as being in the first two families of primitive/primitive twisted torus knots listed in Theorem \ref{thm:pp}. Quick computations show that for the other three families listed in Theorem \ref{thm:pp}, $r$ is not equal to $p \pm q$, so the parameters are not Horadam parameters. However, since twisted torus knots can have multiple sets of parameters representing the same knot type (as demonstrated in Proposition \ref{prop:knottypes}.(\ref{part:type1})), it is possible that these knots could have a different set of Horadam parameters. We leave as an open question whether the set of all Horadam twisted torus knots is disjoint from the collection of all knots in the remaining three families. \\
 
Dean \cite{DeanSFSSurgeries} classifies the types of Seifert knots as being hyper-, middle- and end-Seifert, and he claims a classification of the knots $K(p,q,r,\pm 1, 1)$ that are primitive with respect to $H'$ and middle-Seifert with respect to $H$. Here, we expand and correct Dean's work and talk about how the Horadam twisted torus knots do not generalize to this set.

\begin{theorem}\label{thm:phS}
Let $p$ and $q$ be coprime integers with $p\ge 3$ and $q \ge 2$. The twisted torus knots $K(p,q,r, m, \pm 1)$ with $2 \le r \le p+q$ that are hyper-Seifert fibered with respect to $H$ and primitive with respect to $H'$ are those with $|m|>1$ and $(p,q,r)$ given by the same list of triples as in Theorem \ref{thm:pp}.
\end{theorem}

\begin{proof}
Dean \cite{DeanSFSSurgeries} indicates that the hyper-Seifert fibered twisted torus knots are those that have the $|m|>1$ and $r \equiv \pm 1$ or $\pm q \pmod p$. In other words, Theorem \ref{thm:pp} also lists the values $(p,q,r)$ for which $K(p,q, r, m, \pm1)$, with $|m|>1$, is a primitive/hyper-Seifert twisted torus knot.
\end{proof}

Next, we consider the list of knots that are middle-Seifert with respect to $H$ and primitive with respect to $H'$, similar to Dean's Theorem 4.1. 

\vspace{.1in}

\begin{theorem}\label{thm:pS}
Let $p$ and $q$ be coprime integers with $p\ge 3$ and $q \ge 2$. The twisted torus knots $K(p,q,r, \pm 1)$ with $2 \le r \le p+q$ that are middle-Seifert fibered with respect to $H$ and primitive with respect to $H'$ are given by the following triples $(p,q,r)$:
\begin{enumerate}
\item $(p,q, p-k q)$ where $2 \le k < p/q$,
\item $(p, 3, p+i)$ where $i \in \{1,2\}$ and $p \equiv i \pmod 3$,
\item $(i(2j + 1) + j + \frac{1+\eps}{2}, 2j+1, (i+1)(2j + 1) + \eps)$ where $i, j \ge 1$.
\end{enumerate}
\end{theorem}

We find that in Dean's Theorem 4.1, the inquality $2 \le \beta < p/q$ was forgotten in four of the five cases listed there. Specifically, several cases in that proof consider the option that $q > p/2$, but in that case $p/q <2$, so there are no possible integral values for $\beta$. Hence, only Dean's second case, listed first in Theorem \ref{thm:pS} here, can occur with the hypotheses in Dean's work. 

In his work, Dean restricts to the case  $r \le \max(p,q)$. In the second and third families listed above in Theorem \ref{thm:pS}, $r$ is larger than $p$, so these cases expand Dean's work beyond $r \le \max(p,q)$.

\begin{proof}
Dean \cite{DeanSFSSurgeries} shows that a twisted torus knot $K(p,q,r, \pm 1)$ is primitive with respect to $H'$ if and only if $r \equiv \pm 1$ or $\pm p \pmod q$. Further, the twisted torus knot $K(p,q,r, \pm 1)$ is middle-Seifert with respect to $H$ when $r \equiv \pm \beta q \pmod p$, where $2 \le \beta < \frac{\max(p,q)}{\min(p,q)}$. We have two cases to consider.
\begin{enumerate}[(a)]
\item\label{case:D} $r \equiv \pm p \pmod q$ and $r \equiv \pm \beta q \pmod p$
\item\label{case:A} $r \equiv \pm 1 \pmod q$ and $r \equiv \pm \beta q \pmod p$
\end{enumerate}

Lemma \ref{lem:lee} tells us that we can assume that $p >q$, so as noted above, the inequality $2 \le \beta < p/q$ forces $p > 2q$.

Case \ref{case:D}: The conditions of this case require that $r = kq + \eps p$ and $r = l p + \delta \beta q$, where $\eps = \pm 1$, $\delta = \pm 1$, $\ds k \ge \left\lceil \frac{2-\eps p}{q} \right\rceil$,  and $l$ is a nonzero integer satisfying $\ds l \ge \left\lceil \frac{2-\delta\beta q}{p}\right\rceil$. 

Using $r= kq + \eps p = lp + \delta \beta q$ and the bound on $r$, we have a lower bound on $lp$: $lp \ge 2 - \delta \beta q$. Because $\delta = \pm 1$ and $\beta$ and $q$ are positive, we can say that $2 - \delta \beta q \ge 2 - \beta q$. Further, $\beta < p/q$, so $ 2- \beta q > 2 - p$. Then we have a (fairly loose) lower bound $ lp > 2-p$. In particular, because $l \neq 0$, this tells us that $l \ge 1$. We know that  $lp + \delta \beta q \le p + q$, so $(\delta \beta -1) q \le (1-l)p$. Since $1-l \le 0$ and $\beta, q \ge 2$, we have $\delta = -1$. 

The equality $r= kq + \eps p = lp - \beta q$ tells us that $(k+ \beta )q = (l-\eps)p$. Since $p$ and $q$ are relatively prime,  $l = mq + \eps$ and $k = mp - \beta$ for some integer $m$. Because $l \ge 1$, either $(m,\eps)= (0,1)$ or $ m\ge 1$. If $(m,\eps)= (0,1)$, then $l = 1$ and $k = -\beta$, and we have $(p,q,r) = (p, q, p - |k| q)$, where $2 \le |k| < p/q$. These knots are the first family listed in our theorem, using the more familiar notation of Dean \cite{DeanSFSSurgeries}.

If $m = 1$ and $\eps = -1$, $l = q-1$ and $k = p - \beta$. This gives $r = pq - \beta q - p \le p + q$, which tells us that $p(q-2) \le (\beta +1)q$. This inequality cannot hold when $q\ge 4$, because $2p\le p(q-2)$, but $(\beta +1)q < p+q$ This would require $p<q$, which contradicts our assumptions. When $q = 2$, $r = p - 2\beta$, so we have $(p,q,r) = (p, 2, p - 2\beta)$, where $2 \le \beta < p/2$. This is a special case of the knots listed in the first family of the theorem. 

When $q = 3$, we have $r = 2p - 3\beta $, where $2 \le \beta  <p/3$. To maintain the requirement that $2p- 3\beta \le p + 3$, we need $p \le 3\beta + 3$. On the other hand, $\beta < p/3$, so $3\beta$ must be the element of $\{p-1, p-2, p-3\}$ that is divisible by 3. Since $p$ and $q = 3$ must be relatively prime, $3\beta$ will be either $p-1$ or $p-2$, depending on the congruence class of $p$ modulo 3. 

When $p \equiv 1 \pmod 3$, $3 \beta = p-1$, so $r = p + 1$, and when $p \equiv 2 \pmod 3$, $3 \beta = p-2$, so $r = p + 2$. Then we have the triples $(p,3,p+1)$ with $p \equiv 1 \pmod 3$ and $(p,3,p+2)$ with $p \equiv 2 \pmod 3$. 
These knots constitute the second family listed in the theorem.

If $m \ge 1$ and $\eps = 1$,  we have $r = lp - \beta q = mpq + p - \beta q$. Because $\beta < p/q$, $p - \beta q >0$, so $r \ge mpq$. Then $r$ is strictly greater than $p+q$, and no examples arise from this case.

If $m \ge 2$ and $\eps = -1$, then $kq - p = (mp - \beta) q - p $, which is strictly greater than $(2p -p/q) q - p = 2p(q - 1)$. The expression $2p(q-1)$ is strictly greater than $p+q$ for all allowable values of $p$ and $q$, so no examples arise in this case.

\underline{Case \ref{case:A}:} The conditions of this case require that $r = kq + \eps$ and $ r  = lp + \delta \beta q$, where $\eps = \pm 1$, $\delta = \pm 1$, $k \ge 1$, and $l$ is a nonzero integer satisfying $\ds l \ge \left\lceil \frac{2-\delta\beta q}{p}\right\rceil$.

First, we consider the lower bound on $l$. When $\delta = 1$, because $\beta < p/q$, $\frac{2-\beta q}{p} > \frac{2-p}{p}$. We assume that $p\ge 3$, so $\frac{2-p}{p}$ is greater than $-1$, which tells us that $\left\lceil \frac{2-\beta q}{p}\right\rceil \ge 0$, so the lower bound on $l$ is nonnegative. When $\delta = -1$, we have $l \ge \left\lceil \frac{2+\beta q}{p}\right\rceil$. Because $2 \le \beta$ and $2\le q$, $\frac{2 + \beta q}{p} \ge \frac{6}{p}$. Hence, we find that $l$ will be bounded below by a positive number.

Since $r\le p+q$, we have $kq + \eps \le p+q$, so we know $ (k-1)q + \eps \le p$. Because $2 \le \beta < p/q$,  $ (k-\beta) q + \eps < p$. Then if $\delta = 1$, the equation $r = kq + \eps = lp +\beta q$ requires that $l$ be negative. As noted above $l$ cannot be negative, so we have $\delta = -1$.

Then $r = kq + \eps = lp - \beta q$. This tells us that $lp = (k+ \beta)q + \eps$. We can rewrite the right side of this equation as $(k-1)q + \eps + (\beta + 1) q$, which we know is at most $p + (\beta + 1) q$. Then, using the strict upper bound on $\beta$, we have $l p < 2p + q$. That is, $l \le 2$, and the possible values of $l$ are 1 and 2. 

When $l =1$, we have $r = kq + \eps$ and $p = (k+ \beta) q + \eps$, so we see that $r = p - \beta q$. This is a special case of an already-listed family of knots. When $l = 2$, we have $2p = (k+\beta) q + \eps$, so we see that $k+ \beta$ and $q$ are both odd. Let $k+ \beta = 2i + 1$ and $q = 2j + 1$ where $i, j \ge 1$. Then $p = i(2j + 1) + j + \frac{1+ \eps}{2}$ and $ r= (2i+1 - \beta)(2j + 1) + \eps$. The bounds on $\beta$ guarantee that $2 \le \beta \le i$ because $\frac{p}{q} = i + \frac{j + \frac{1+\eps}{2}}{2j+1}$, which is less than $i+1$. To maintain the requirement that $r \le p+q$, then, we need $(2i+1 - \beta)(2j + 1) + \eps \le (i+1)(2j + 1) + j + \frac{1+ \eps}{2}$. Reducing this inequality, we see that we need $(i- \beta)(2j+1) \le j + \frac{1-\eps}{2}$, which is only possible if $ \beta = i$. Thus we have the third family of knots listed, and the proof is complete.

\end{proof}

We recall that an $(m,n; a,b)$-Horadam sequence is the sequence $\{H_0,H_1,H_2, \ldots\}$ where $H_0 = m$, $H_1 = n$ and $H_k = aH_{k-2} + bH_{k-1}$ for $k \ge 2$. We can see that the first family of knots listed in Theorem \ref{thm:pS} can be considered either as a twisted torus knot $K(H_1, H_0, H_2, -1)$ in the $(q,p; -k, 1)$-Horadam twisted torus knots or as $K(H_2, H_1, H_0, -1)$ in  the $(p-kq, q; 1, k)$-Horadam twisted torus knots. As with the list of primitive/primitive twisted torus knots in Theorem \ref{thm:pp}, the other families of knots in Theorem \ref{thm:pS} are not apparently Horadam twisted torus knots, so it would be interesting to completely classify the Horadam twisted torus knots in these families.

 \section{Acknowledgements}

The author would like to thank Cameron Gordon for helpful comments on this work. Additionally, the author is especially thankful to the referee for careful reading of this work and constructive comments that produced a better product.

%
%

\bibliographystyle{hplain}
\bibliography{HoradamTTKs}

\end{document}